\documentclass[10pt,a4paper]{amsart}

\usepackage{amsmath,amssymb,amsthm,amsfonts}
\usepackage{amsbsy}
\usepackage[utf8]{inputenc}
\usepackage[T1]{fontenc}
\usepackage{enumitem}
\usepackage{a4wide}
\usepackage[english]{babel}
\usepackage{longtable}
\usepackage{url}
\usepackage{subfig}
\usepackage{aliascnt}

\usepackage{tikz}
\usetikzlibrary{matrix}
\usetikzlibrary{arrows}
\usetikzlibrary{fit}
\usetikzlibrary{shapes}

\usepackage{booktabs}

\usepackage{young}
\usepackage[enableskew]{youngtab}
\usepackage{ytableau}

\newtheorem{theorem}{Theorem}[section]
\newaliascnt{corollary}{theorem}

\aliascntresetthe{corollary}
\newaliascnt{lemma}{theorem}
\newtheorem{lemma}[lemma]{Lemma}
\aliascntresetthe{lemma}
\newaliascnt{proposition}{theorem}
\newtheorem{proposition}[proposition]{Proposition}
\aliascntresetthe{proposition}
\newtheorem*{zc*}{Zassenhaus Conjecture (ZC)}
\newtheorem*{pq*}{Prime Graph Question (PQ)}
\newtheorem*{theorem*}{Theorem}

\theoremstyle{definition}

\newaliascnt{definition}{theorem}
\newtheorem{definition}[definition]{Definition}
\aliascntresetthe{definition}
\theoremstyle{remark}
\newaliascnt{example}{theorem}
\newtheorem{example}[example]{Example}
\aliascntresetthe{example}
\newaliascnt{remark}{theorem}
\newtheorem{remark}[remark]{Remark}
\aliascntresetthe{remark}
\newaliascnt{notation}{theorem}

\aliascntresetthe{notation}

\definecolor{Schweinchenrosa}{HTML}{FFAAAA}

{\title{On the Prime Graph Question for Integral Group Rings of Conway simple groups}
\author{Leo Margolis}
\address{Vrije Universiteit Brussel, Faculty of Sciences, Department of Mathematics, Pleinlaan 2, B-1050 Brussels, Belgium}
\email{leo.margolis@vub.be}
\thanks{This research was partially supported by the European Commission individual Marie Curie-Sk\l odowska Fellowship in the H2020 program under Grant 705112-ZC and the FWO (Research Foundation Flanders)}

\newcommand\blfootnote[1]{%
	\begingroup
	\renewcommand\thefootnote{}\footnote{#1}%
	\addtocounter{footnote}{-1}%
	\endgroup
}

\begin{document}

\maketitle

\begin{abstract} The Prime Graph Question for integral group rings asks if it is true that if the normalized unit group of the integral group ring of a finite group $G$ contains an element of order $pq$, for some primes $p$ and $q$, also $G$ contains an element of that order. We answer this question for the three Conway sporadic simple groups after reducing it to a combinatorial question about Young tableaux and Littlewood-Richardson coefficients. This finishes work of V. Bovdi, A. Konovalov and S. Linton.
\end{abstract}


\section{Introduction}

The unit group of the integral group ring $\mathbb{Z}G$ of a finite group $G$ has given rise to many interesting research questions and results, as recorded e.g. in the monographs \cite{Sehgal1993, JespersDelRioGRG}. Many of these questions concern the connection between finite subgroups of units of $\mathbb{Z}G$ and the group base $G$, e.g. the Isomorphism Problem which was open for over 50 years until answered by Hertweck \cite{HertweckAnnals} or the recently answered Zassenhaus Conjecture \cite{EiseleMargolis17} which was open for over 40 years. For both questions there turned out to be counterexamples. In this light the weaker versions of the former Zassenhaus Conjecture, which are still open, become more important. One of these versions, the so-called Prime Graph Question, gained attention after being introduced by Kimmerle in \cite{KimmiPQ}. To formulate it denote by $\mathrm{V}(\mathbb{Z}G)$ the group of the so-called \emph{normalized units} in $\mathbb{Z}G$, i.e. the units whose coefficients sum up to $1$.\\

\textbf{Prime Graph Question:} Let $G$ be a finite group and $p$ and $q$ some primes. If $\mathrm{V}(\mathbb{Z}G)$ contains an element of order $pq$, does $G$ contain an element of order $pq$?\\

The \emph{prime graph} $\Gamma(G)$ of a group $G$ is a graph whose vertices are the primes appearing as orders of elements in $G$ and two vertices $p$ and $q$ are connected if and only if $G$ contains an element of order $pq$. Hence the Prime Graph Question can also be formulated as: For $G$ a finite group, does $\Gamma(G) = \Gamma(\mathrm{V}(\mathbb{Z}G))$ hold? 

The Prime Graph Question seems particularly approachable, since, in contrast to other questions on the finite subgroups in $\mathrm{V}(\mathbb{Z}G)$, a reduction result is available here: The Prime Graph Question holds for a group $G$, if it holds for all almost simple images of $G$ \cite[Theorem 2.1]{KimmiKonovalov17}. Recall that a group $A$ is called \emph{almost simple} if it is isomorphic to a subgroup of the automorphism group of a non-abelian simple group $S$ containing the inner automorphisms of $S$, i.e. $S \cong {\rm{Inn}}(S) \leq A \leq {\rm{Aut}}(S)$. In this case $S$ is called the socle of $A$.

It is known that the Prime Graph Question has a positive answer for almost simple groups with socle isomorphic to a projective special linear group $\operatorname{PSL}(2,p)$ or $\operatorname{PSL}(2,p^2)$ for $p$ a prime \cite[Theorem A]{4primaryI} or an alternating group of degree up to $17$ \cite{BaechleCaicedoSymmetric}. The question also has a positive answer for groups whose order is divisible by exactly three pairwise different primes \cite{KimmiKonovalov17, Gitter} and many almost simple groups whose order is divisible by four pairwise different primes \cite{4primaryII}.

Employing a computer implementation of a method introduced by Luthar and Passi \cite{LutharPassi1989} and Hertweck \cite{HertweckBrauer}, nowadays known as the HeLP method, in a series of papers between 2007 and 2012 Bovdi, Konovalov and different collaborators of them studied the Prime Graph Question for sporadic simple groups \cite{BovdiKonovalovM12, BovdiKonovalovM11, BovdiKonovalovMcL, BovdiKonovalovM22, BovdiKonovalovSuz, BovdiKonovalovM23, BovdiKonovalovONan, BovdiKonovalovRu, BovdiKonovalovHS, BovdiKonovalovJ, BovdiKonovalovConway, BovdiKonovalovM24}. Overall they studied 17 sporadic simple groups and were able to prove the Prime Graph Question for 13 of these groups, an overview of their results can be found in \cite[Section 5]{KimmiKonovalov15}. Also in \cite[Corollary 5.4]{KimmiKonovalov15} it was recorded that the Prime Graph Question holds for almost simple groups with a socle isomorphic to one of the $13$ sporadic simple groups for which Bovdi, Konovalov et al. proved the Prime Graph Question. The four groups studied by Bovdi, Konovalov et al. for which they were not able to prove the Prime Graph Question were the O'Nan simple group \cite{BovdiKonovalovONan} and the three Conway simple groups \cite{BovdiKonovalovConway}. Developing further the combinatorial side of a method introduced in \cite{Gitter} in this note we obtain a positive answer for the Prime Graph Question for the latter groups. This is the first contribution to the Prime Graph Question for sporadic simple groups since the papers of Bovdi, Konovalov et al.

\begin{theorem}\label{TheoremConway}
The Prime Graph Question has a positive answer for the sporadic Conway simple groups $Co_3$, $Co_2$ and $Co_1$. 
\end{theorem}

Note that the outer automorphism group of each Conway simple group is trivial, so Theorem~\ref{TheoremConway} proves the Prime Graph Question also for every almost simple group whose socle is isomorphic to a Conway simple group.

To describe our main tool let us first introduce some notation. Let $G$ be a finite group, $u \in \mathrm{V}(\mathbb{Z}G)$ a unit of order $n$ and $\chi$ an ordinary character of $G$ with corresponding representation $D$. Then we can extend $D$ linearly to $\mathbb{Z}G$ and afterwards restrict it to $\mathrm{V}(\mathbb{Z}G)$, obtaining an extension of $D$ and $\chi$. In particular $D(u)$ is then a matrix of finite order dividing $n$. Let $ \xi$ be an $n$-th root of unity. Then we denote by $\mu(\xi, u, \chi)$ the multiplicity of $\xi$ as an eigenvalue of $D(u)$. Moreover for an integer $k$ we will denote by $\zeta_k$ a primitive $k$-th root of unity. Then the results given above will follow from an application of the following theorem. 

\begin{theorem}\label{MainProp}
Let $G$ be a finite group, $p$ an odd prime and $m$ a positive integer not divisible by $p$. Assume that some $p$-block of $G$ is a Brauer Tree Algebra with ordinary characters $\chi_1,$ ...,$\chi_p$ and with the Brauer tree being a line of the form
\[
\begin{tikzpicture}
\node[label=north:{$\chi_1$}] at (0,1.5) (1){};
\node[label=north:{$\chi_{2}$}] at (1.5,1.5) (2){};
\node[label=north:{$\chi_{3}$}] at (3,1.5) (3){};
\node[label=north:{$\chi_{p-2}$}] at (4.5,1.5) (4){};
\node[label=north:{$\chi_{p-1}$}] at (6,1.5) (5){};
\node[label=north:{$\chi_{p}$}] at (7.5,1.5) (6){};
\foreach \p in {1,2,3,4, 5, 6}{
\draw[fill=white] (\p) circle (.075cm);
}
\draw (.075,1.5)--(1.425,1.5);
\draw (1.575,1.5)--(2.925,1.5);
\draw[dashed] (3.075,1.5)--(4.425,1.5);
\draw (4.575,1.5)--(5.925,1.5);
\draw (6.075,1.5)--(7.425,1.5);
\end{tikzpicture}.
\] 
Moreover assume that in the rings of values of $\chi_1$, ...,$\chi_p$ the prime $p$ is unramified.

Let $u \in \mathrm{V}(\mathbb{Z}G)$ be a unit of order $pm$ and let $\xi$ be some $m$-th root of unity.
Then the inequality
\[\mu(\xi \cdot \zeta_p, u, \chi_{p-1}) - \mu(\xi \cdot \zeta_p, u, \chi_p) \leq \mu(\xi, u, \chi_1) - \sum_{i=1}^{p-2} (-1)^i \mu(\xi \cdot \zeta_p, u, \chi_i) \]
holds.
\end{theorem} 
This theorem might be regarded as a generalization of \cite[Proposition 3.2]{4primaryII}. The condition on the unit can also be formulated in terms of character values which makes the inequality easier to check by hand, although it becomes much longer then, cf. Lemma~\ref{LemCharVals}.

Theorem~\ref{MainProp} will be achieved using the so-called lattice method introduced in \cite{Gitter} and developed further in \cite{4primaryII}. It is typical for the study of group rings that different fields of mathematics, such as group theory, ring theory, representation theory and number theory are combined to achieve results. In the present paper the main tool is combinatorics, more precisely calculations with Young tableaus and Littlewood-Richardson coefficients, cf. the proof of Theorem~\ref{MainProp}. It is quite plausible that Theorem~\ref{MainProp}, which is the main tool to prove Theorem~\ref{TheoremConway}, can be applied also to study the Prime Graph Question or related questions for many other groups. Also the methods presented here could be used to prove variations of Theorem~\ref{MainProp}, in particular for other forms of Brauer trees. 

We will also need to use the HeLP method to obtain enough restrictions on torsion units which allow the application of Theorem~\ref{MainProp}. We will start by recalling the needed methods in Section 2, develop the necessary combinatorics in Section 3, continue to prove Theorem~\ref{MainProp} and finally show how this can be applied to prove Theorem~\ref{TheoremConway} in Section 4.

\section{Preliminaries and methods}
A fundamental notion when studying torsion units of integral group rings are so-called partial augmentations. Let $u \in \mathbb{Z}G$ be an element of the form $\sum_{g \in G}z_g g$ and denote by $x^G$ the conjugacy class of an element $x$ in $G$. Then $\varepsilon_{x^G}(u) = \sum_{g \in x^G} z_g$ is called the \emph{partial augmentation} of $u$ at $x$. A fundamental theorem of Hertweck \cite[Theorem 2.3]{HertweckBrauer} states that if $u \in \mathrm{V}(\mathbb{Z}G)$ is a unit of order $n$ then $\varepsilon_{x^G}(u) \neq 0$ implies that the order of $x$ divides $n$. This implies in particular that the exponents of $G$ and $\mathrm{V}(\mathbb{Z}G)$ coincide, a result due originally to Cohn and Livingstone \cite{CohnLivingstone}. Moreover $\varepsilon_1(u) = 0$ unless $u = 1$ by the Berman-Higman Theorem \cite[Proposition 1.5.1]{JespersDelRioGRG}.

\subsection{The HeLP method}

Partial augmentations are class functions and as such can be investigated using representation theory. Note that if $\chi$ is an ordinary character of $G$ and $u \in \mathrm{V}(\mathbb{Z}G)$ a torsion unit of order $n$ then $\chi(u) = \sum_{x^G} \varepsilon_{x^G}(u)\chi(g)$, where the sum runs over the conjugacy classes of $G$. It was shown by Hertweck \cite[Theorem 3.2]{HertweckBrauer} that this also holds for $p$-Brauer characters, if $p$ is not a divisor of $n$ and the sum is  understood to run on $p$-regular conjugacy classes of $G$. On the other hand if we are given the partial augmentations of $u$ and its powers we can compute the eigenvalues, including multiplicities, of $u$ under any ordinary or $p$-modular representation of $G$, again assuming $p$ is not dividing $n$. These multiplicities can be expressed in explicit formulas, depending on the partial augmentations of $u$ and its powers, and that the results of these formulas are non-negative integers is the basic idea of the HeLP method, cf. \cite[Section 2]{4primaryI} for a detailed explanation and \cite{HertweckBrauer} for the proofs.

Hence there is an algorithmic method which allows to obtain restrictions on the possible partial augmentations of torsion units in $\mathrm{V}(\mathbb{Z}G)$, provided we have some knowledge on the characters of $G$. If we apply the HeLP method to the whole character table of $G$ then, for any fixed $n$, we will obtain a finite number of possibilities for the partial augmentations of units of order $n$ in $\mathrm{V}(\mathbb{Z}G)$. If it happens that we obtain no possibility for units of some given order $n$, then we know that there exist no units at all in $\mathrm{V}(\mathbb{Z}G)$ of order $n$. This is the basic idea in the application of the HeLP method for the study of the Prime Graph Question and in this way Bovdi, Konovalov et al. gave their proofs for 13 sporadic groups. We will also rely on the HeLP method to produce a finite number of possibilities for the partial augmentations of units of certain orders and then Theorem~\ref{MainProp} will do the rest. We will apply here a computer implementation of the HeLP method as a package \cite{HeLPPackage} in the computer algebra system \texttt{GAP} \cite{GAP}. This package can also be used to reproduce the results of Bovdi, Konovalov et al.

\subsection{The lattice method}\label{sec:LatMet}

The lattice method is an other method which makes use of the multiplicities of eigenvalues of a torsion unit $u \in \mathrm{V}(\mathbb{Z}G)$ of order $n$ under ordinary representations of $G$. This method was introduced in \cite{Gitter} and the basic idea consists in obtaining restrictions on the simple $kG$-modules when viewed as $k\langle u \rangle$-modules, where $k$ is a field of characteristic $p$ for a divisor $p$ of $n$. These simple $kG$-modules are the $p$-modular composition factors of ordinary representations of $G$ and studying different ordinary representations with common composition factors can finally produce a contradiction to the existence of $u$.

We will only recall those parts of the lattice method necessary for our means in this article. We will cite the articles \cite{Gitter, 4primaryII} where this method is explained, but many facts can be found in text books on representation theory. First of all let $C = \langle c \rangle$ be a cyclic group of order $pm$ where $p$ is a prime and $m$ an integer not divisible by $p$. Let $k$ be a field of characteristic $p$ containing a primitive $m$-th root of unity $\xi$. Then there are, up to isomorphism, exactly $pm$ indecomposable $kC$-modules determined by a pair $(i,j)$ where $1 \leq i \leq p$, $1 \leq j \leq m$ such that a module $M$ corresponding to $(i,j)$ is $i$-dimensional as a $k$-vector space and $c^p$ acts as $\xi^j$ on $M$. We denote this module by $I_i^j$ and if the action of $c^p$ is clear from the context or if $m = 1$ we simply write $I_i$. The module $I_i^j$ is simple if and only if $i = 1$ and each $I_i^j$ is uniserial \cite[Proposition 2.2]{Gitter}.

Now assume that $M$ is a $kC$-module, $d$-dimensional as $k$-vector space, such that $c^p$ acts as $\xi^j$ on $M$ for some fixed $j$. Then by the above we know that $M \cong a_pI_p \oplus a_{p-1}I_{p-1} \oplus \dots \oplus a_1I_1$ for certain non-negative integers $a_1, \ ..., \ a_p$. Moreover $a_pp + a_{p-1}(p-1) + ... + a_1 = d$ and hence 
\[\lambda = (\underbrace{p,...,p}_{a_p}, \underbrace{p-1,...,p-1}_{a_{p-1}}, \dots, \underbrace{1,...,1}_{a_1}) \]
is a partition of $d$. We call $\lambda$ the \emph{partition corresponding to $M$}.

A common framework to work with partitions using combinatorics is provided by considering Young tableaux and several notions connected with these tableaux. 
We recall here quickly those of them we will need in this paper. They all can be found in several textbooks on algebraic combinatoric as e.g. \cite{Fulton}. A \emph{Young diagram} corresponding to a partition $\lambda = (\lambda_1, \ ..., \ \lambda_r)$ is a diagram consisting of boxes ordered in rows and columns such that the first row contains $\lambda_1$ boxes, the second row $\lambda_2$ boxes etc. A \emph{skew diagram} is obtained by removing a smaller Young diagram from a larger one that contains it. I.e. if $\mu = (\mu_1,\ ..., \ \mu_r) \leq \lambda = (\lambda_1, \ ..., \ \lambda_r)$, meaning that $\mu_i \leq \lambda_i$ for all $i$, then the skew diagram corresponding to $\lambda/\mu$ is obtained by erasing from the Young diagram corresponding to $\lambda$, always from left to right, $\mu_1$ boxes in the first row, $\mu_2$ boxes in the second row etc. If we fill a skew diagram with entries from an alphabet (in our case the alphabet always will be the positive integers) a skew diagram becomes a \emph{skew tableau} $T$. By "filling" we mean writing a letter in every box. $T$ is called \emph{semistandard} if reading a row from left to right the entries do not decrease and reading a column from top to bottom the entries strictly increase. 

For a box $b$ in a skew tableau $T$ we denote by $w(b)$ the word which we obtain reading $T$ from top to bottom and from right to left until $b$, including the entry in $b$. We denote by $w(T)$ the word which we obtain reading all boxes of $T$ in this way, i.e. if $b$ is the lowest box in the first (i.e. most left) column of $T$ then $w(b) = w(T)$. We say that $T$ satisfies the \emph{lattice property} if for every box $b$ in $T$ the word $w(b)$ contains the letter $1$ at least as many times as the letter $2$, the letter $2$ at least as many times as $3$ etc. If the maximal letter of $w(T)$ is $s$ and $w(T)$ contains $\nu_1$ times the letter $1$, $\nu_2$ times the letter $2$ etc. then $\nu = (\nu_1, \ ..., \ \nu_s)$ is called the \emph{content} of $T$. Note that if $T$ is a skew tableau satisfying the lattice property then the content  $\nu = (\nu_1, \ ..., \ \nu_s)$ of $T$ is a partition of $\nu_1 + ... + \nu_s$ as then $\nu_1 \geq \nu_2 \geq ... \geq \nu_s$.
If now $\nu$ is some partition then the \emph{Littlewood-Richardson coefficient} $c^\lambda_{\mu, \nu}$ is the number of semistandard skew tableaux corresponding to $\lambda/\mu$ which satisfy the lattice property and have content $\nu$. A fact of fundamental importance to us is that the Littlewood-Richardson coefficient $c^\lambda_{\mu, \nu}$ is symmetric in $\mu$ and $\nu$.

These combinatorial objects play a role in the study of torsion units of group rings via the following observation \cite[Theorem 2.8]{4primaryII}.

\begin{theorem}\label{ThLRCoeff} Let $C = \langle c \rangle$ be a cyclic group of order $p$, for $p$ a prime, and $k$ a field of characteristic $p$. Let $M, U$ and $Q$ be $kC$-modules with corresponding partitions $\lambda, \mu$ and $\nu$ respectively. Then $M$ contains a submodule $\tilde{U}$ isomorphic to $U$ such that $M/\tilde{U} \cong Q$ if and only if $c^\lambda_{\mu, \nu} \neq 0$.
\end{theorem} 

To pass from multiplicities of eigenvalues under ordinary representations to modules over modular group algebras we will employ the following \cite[Propositions 2.3, 2.4]{Gitter}.

\begin{proposition}\label{PropChar0} Let $C = \langle c \rangle$ be a cyclic group of order $pm$ such that $p$ does not divide $m$. Let $R$ be a complete local ring of characteristic $0$ containing a primitive $m$-th root of unity $\xi$ such that $p$ is contained in the maximal ideal of $R$ and not ramified in $R$. Denote by $k$ the residue class field of $R$ and adopt the bar-notation for reduction modulo the maximal ideal of $R$.

Let $D$ be a representation of $C$ such that the eigenvalues of $D(u)$ in the algebraic closure of the quotient field of $R$, with multiplicities, are $\xi A_1, \xi^2 A_2, \dots, \xi^{m} A_{m}$ for certain multisets $A_i$ consisting of $p$-th roots of unity. Here also $A_i = \emptyset$ is possible. Let $\zeta$ be a non-trivial $p$-th root of unity. Note that since the sum of the eigenvalues of $D(u)$ is an element in $R$ we know for every $i$ that if $A_i$ contains $\zeta$ exactly $r$ times then $A_i$ contains also $\zeta^2$, ...,$\zeta^{p-1}$ exactly $r$ times.

Let $M$ be an $RC$-lattice affording the representation $D$. Then 
\[M \cong M_1 \oplus M_2 \oplus ... \oplus M_m \]
such that for every $i$ we have: If $A_i$ contains $\zeta$ exactly $r$ times and $1$ exactly $s$ times then
\[\overline{M_i} \cong aI^i_p \oplus (r-a)I^i_{p-1} \oplus (s-a)I^i_1 \]
for some non-negative integer $a \leq {\rm{min}}\{r,s\}$.
\end{proposition}

For $u \in \mathrm{V}(\mathbb{Z}G)$ a unit of order $pm$ for which we know the partial augmentations, including those of its powers, these facts allow to derive information on the isomorphism type of simple $kG$-modules, for $k$ a big enough field of characteristic $p$, from the eigenvalues of $u$ under ordinary representations of $G$, see the proof of Theorem~\ref{MainProp}, or \cite{Gitter} for a more detailed sketch of the method.

\subsection{A reformulation using character values}

The multiplicities in Theorem~\ref{MainProp} can be computed from the partial augmentations of $u$ and its proper powers. This can be done by hand or using the \texttt{GAP}-package \texttt{HeLP} \cite{HeLPPaper}, more precisely the command \texttt{HeLP\_MultiplicitiesOfEigenvalues}. The theorem can however also be translated in a condition on character values, so that the condition of Theorem~\ref{MainProp} can be more easily checked without using \texttt{GAP} (or just to look up the character values). This is particularly handy if the characters in question only have integral values and the following lemma provides a way to do this.

\begin{lemma}\label{LemCharVals} Let $G$ be a finite group and $u \in \mathrm{V}(\mathbb{Z}G)$ a unit of order $pq$ where $p$ and $q$ are different primes. Let $\chi$ be a character of $G$ which takes only integral values. Set $\chi(1) = d$, $\chi(u^p) = x$, $\chi(u^q) = y$ and $\chi(u) = z$. Then the following formulas hold.
\begin{align*}
pq \cdot \mu(1, u, \chi) &= d  + (q-1)x + (p-1)y + (p-1)(q-1)z, \\
pq \cdot \mu(\zeta_p, u, \chi) &= d  + (q-1)x  - y - (q-1)z, \\
pq \cdot \mu(\zeta_q, u, \chi) &= d  - x + (p-1)y - (p-1)z, \\
pq \cdot \mu(\zeta_{pq}, u, \chi) &= d  - x - y + z.
\end{align*}
\end{lemma}

\begin{proof}
Recall that the sum over all primitive $p$-th or $q$-th roots of unity is $-1$ while the sum over all primitive $pq$-th roots of unity is $1$. Note that for $i$ not divisible by $p$ we have $\mu(\zeta_p, u, \chi) = \mu(\zeta_p^i, u, \chi)$, for $i$ not divisible by $q$ we have $\mu(\zeta_q, u, \chi) = \mu(\zeta_q^i, u, \chi)$ and for $i$ not divisible by $pq$ also $\mu(\zeta_{pq}, u, \chi) = \mu(\zeta_{pq}^i, u, \chi)$. Denote by $D$ a representation corresponding to $\chi$. So if e.g. $\mu(\zeta_p, u, \chi) = k$, i.e. $\zeta_p$ appears $k$ times as an eigenvalue of $D(u)$, then also each other primitive $p$-th root of unity appears $k$ times and together they contribute $-k$ to $\chi(u) = \operatorname{tr}(D(u))$. Moreover thinking of $D(u)$ as a diagonalized matrix and taking the $p$-th power in $D(u^p)$ each of these eigenvalues becomes $1$ and so they contribute $(p-1) \cdot k$ to $\chi(u^p)$.   

Arguing in the same manner for primitive $pq$-th roots of unity we obtain that they contribute $\mu(\zeta_{pq}, u, \chi)$ to $\chi(u)$, while contributing e.g. $(q-1)\mu(\zeta_{pq}, u, \chi)$ to $\chi(u^q)$ as they become primitive primitive $p$-roots of unity as eigenvalues of $D(u^q)$.

Employing the same arguments for all roots of unity involved in $D(u)$ we obtain:
\begin{align*}
\chi(1) &= \mu(1, u, \chi) + (p-1)\mu(\zeta_p, u, \chi) + (q-1)\mu(\zeta_q, u, \chi) + (p-1)(q-1)	\mu(\zeta_{pq}, u, \chi), \\
\chi(u^q) &=  \mu(1, u, \chi) - \mu(\zeta_p, u, \chi) + (q-1)\mu(\zeta_q, u, \chi) - (q-1)\mu(\zeta_{pq}, u, \chi), \\
\chi(u^p) &= \mu(1, u, \chi) + (p-1)\mu(\zeta_p, u, \chi) - \mu(\zeta_q, u, \chi) - (p-1)\mu(\zeta_{pq}, u, \chi), \\
\chi(u) &= \mu(1, u, \chi) - \mu(\zeta_p, u, \chi) - \mu(\zeta_q, u, \chi) + \mu(\zeta_{pq}, u, \chi).
\end{align*}
The lemma now follows by linear transformations.
\end{proof}

\section{Combinatorics}

The main result of this section, which is the main ingredient in the proof of Theorem~\ref{MainProp}, is Proposition~\ref{MinMax}. It provides bounds on the possible entries in a semistandard skew tableau satisfying the lattice property for a special form of skew tableaux which are typical for our intended application.

\begin{definition}\label{Def form A} We call a skew diagram of form $\mathcal{A}$ if it consists of $p$ columns such that the second to $(p-1)$-th column all end at the same row, say the $m$-th row. 
This is equivalent to saying that there is a number $m$ such that for a box $b$ in the first to $(p-2)$-th column and in the $k$-th row, for $k \leq m$, there is always a box to the right of $b$ and moreover the second column contains no box in the $(m+1)$-th row. Furthermore we also allow the $p$-th column to be empty. 

We call the boxes in the first to $(p-1)$-th column which lie in the $m$-th row and above the \textit{body} of the skew diagram. The boxes of the first column lying below the $m$-th row are called the \textit{tail} and the $p$-th column the \textit{head} of the tableau.

For a semistandard skew tableau $T$ satisfying the lattice property we denote by $\gamma_i(T)$ the number of letters which appear at least $i$ times in $w(T)$. 
\end{definition}

\begin{example} In Figure~\ref{Fig Def 1} the upper skew diagram is of form $\mathcal{A}$ while the lower one is not. In the first skew diagram the boxes in the head are marked with $h$, the boxes of the body with $b$ and the boxes of the tail with $t$.

\begin{figure}[h]
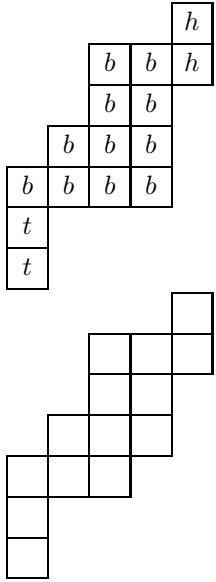

\begin{ytableau}
\none & \none & \none & \none & h \\
\none & \none & b & b & h \\
\none & \none & b & b \\
\none & b & b & b \\
b & b & b & b \\
t \\
t
\end{ytableau}

\begin{ytableau} \none & \none & \none & \none &  \\
\none & \none &  & & \\
\none & \none &  &  \\
\none &  &  &  \\
 &  & \\
 \\ \\
 \end{ytableau} 
{\caption{A skew diagram of form $\mathcal{A}$ and not of form $\mathcal{A}$}\label{Fig Def 1}}
\end{figure}

\end{example}

\textbf{Remark:} We clarify our way of speaking about the location of a box in a skew tableau relative to another box, so that the following proofs become more readable. Let $b$ be a box in a skew tableau $T$. If we speak about the box lying right from $b$ we mean the unique box neighbouring $b$ on the right. This box will be denoted by $b_r$ then. If we speak about a box lying right from $b$ this can be any box lying in a column of $T$ which lies to the right of the column containing $b$. We will also use a different notation for these kind of boxes.

We speak the same way of boxes lying to the left, above or below $b$.

\begin{lemma}\label{AppearsAtLeast}
Let $T$ be a semistandard skew tableau of form $\mathcal{A}$ with $p$ columns satisfying the lattice property. Let $b$ be a box in the $k$-th column lying in the body of $T$ with entry $e$ where $1 \leq k \leq p-1$. Then $w(b)$ contains $e$ at least $(p-k)$ times. 
\end{lemma}

\begin{proof}
We will argue by decreasing induction. For $k = p-1$ the claim is clear, since being the entry in $b$ the letter $e$ appears in $w(b)$ at least once. So let $k < p-1$ and let $e_r$ be the entry in the box $b_r$ right from $b$. Note that $b_r$ exists by the assumption that $T$ is of form $\mathcal{A}$ and $b$ a box in the body of $T$. Then by induction the letter $e_r$ appears at least $p-k-1$ times in $w(b_r)$. As $e \leq e_r$, since $T$ is semistandard, and $w(b_r)$ satisfies the lattice property also $e$ must appear at least $p-k-1$ times in $w(b_r)$. Hence $e$ appears at least $p-k$ times in $w(b)$. 
\end{proof}

\begin{lemma}\label{AtMosth}
Let $T$ be a semistandard skew tableau of form $\mathcal{A}$ with $p$ columns satisfying the lattice property. Let $1 \leq k \leq p-2$ and let $b$ be a box in the $k$-th column of $T$ inside the body of $T$ with entry $e$. Assume that the box right from $b$ is the $h$-th box in the $(k+1)$-th column of $T$. Then $e \leq h$.
\end{lemma}

\begin{proof}
Assume that reading from top to bottom and right to left $b$ is the first box contradicting the claim. In particular, $e > h$. Let $b_r$ be the box right from $b$ with entry $e_r$.
We will show that the $(k+1)$-th column contains every letter $1$, $2$, ...,$e-1$. As $T$ is semistandrad the entries in the $(k+1)$-th column of $T$ are strictly increasing when read from top to bottom. So assuming that the column contains all entries between $1$ and $e-1$ the entry in the first box of the column must in fact be $1$, the entry in the second box $2$, ..., the entry in the $(e-1)$-th box must be $e-1$. In particular, as $e > h$, the $h$-th box will contain $h$, meaning $h = e_r  < e$, contradicting that $T$ is a semistandard tableau. So once we show that the $(k+1)$-th column contains every letter $1$, $2$, ..., $e-1$ the lemma will follow by contradiction.

We will show also that when $b_{k+1}$ is a box in the $(k+1)$-th column containing an entry $e_s$ smaller than $e$ then this entry appears exactly $p-k$ times in $w(b_{k+1})$ and if the box left from $b_{k+1}$ exists then the entry of this box is strictly smaller than $e_s$. 

We will argue by decreasing induction. So we will first show that the $(k+1)$-th column contains the letter $e-1$, that if $b_{k+1}$ denotes the box in the $(k+1)$-th column containing $e-1$, then $e-1$ appears in $w(b_{k+1})$ exactly $p-k$ times and if the box left from $b_{k+1}$ exists its entry is strictly smaller than $e-1$. 
By Lemma~\ref{AppearsAtLeast} the word $w(b)$ contains the letter $e$ at least $p-k$ times and so $w(b)$ also contains every letter smaller than $e$ at least $p-k$ times, as $w(b)$ satisfies the lattice property. By assumption $b$ is the first box violating the lemma when going through the tableau top to bottom and right to left. So if the box $b_a$ above $b$ exists and has entry $e_a$ then the box right from $b_a$ is the $(h-1)$-th box in the $(k+1)$-th column. So $e_a \leq h - 1$ and $e - e_a >  h - (h-1) = 1$. In any case, if $b_a$ exists or not, we have that the $k$-th column does not contain the entry $e-1$, as the entries in the $k$-th column must strictly increase and there are no boxes between the entries $e_a$ and $e$. Note that then also $e-1$ is not the entry in a box lying above $b$ and to the left of $b$, since this would be a box contradicting the fact that $T$ is semistandard. 
Since $w(b)$ contains $e-1$ at least $p-k$ times and $w(b)$ does not take into account boxes lying below $b$ as well as boxes lying left from $b$ and not above $b$, we conclude that every column right from $b$ must contain $e-1$. In particular, the $(k+1)$-th column. Moreover if $b_{k+1}$ is the box in the $(k+1)$-th column containing $e-1$ then $e-1$ appears exactly $p-k$ times in $w(b_{k+1})$, as the boxes between $b_{k+1}$ and $b$ (counted right to left and top to bottom) do not contain the letter $e-1$. Furthermore $b_{k+1}$ is a box above $b_r$, since otherwise $e_r \leq e-1 < e$, which would contradict the fact that $T$ is semistandard. Since the $k$-th column does not contain $e-1$ the box lying left from $b_{k+1}$, if it exists, must contain an entry smaller than $e-1$, since $T$ is semistandard. This proofs the base case of our induction.

The arguments for the induction step are very similar to the arguments for the base case.
So assume that $e_s \leq e-1$ is an entry contained in the box $b_{k+1}$ which lies in the $(k+1)$-th column, that $e_s$ appears exactly $p-k$ times in $w(b_{k+1})$ and that the box left from $b_{k+1}$, if existent, contains an entry smaller than $e_s$. Denote the box left from $b_{k+1}$, if existent, by $(b_{k+1})_l$. Then $(b_{k+1})_l$ contains an entry strictly smaller than $e_s$ by induction. So the the box above $(b_{k+1})_l$, if existent, has maximal entry $e_s-2$, as $T$ is semistandard. In particular $e_s-1$ is not an entry in a box lying above and to the left of $b_{k+1}$. But, by Lemma~\ref{AppearsAtLeast}, the word $w(b_{k+1})$ contains $e_s$ at least $p-k$ times, so it contains $e_s-1$ at least $p-k$ times and we get that $e_s-1$ must be an entry in the $(k+1)$-th column. So it must lie in the box $(b_{k+1})_a$ above $b_{k+1}$. Then $e_s-1$ appears exactly $p-k$ times in $w((b_{k+1})_a)$, as it does not appear in a box above and to the left of $(b_{k+1})_a$. Also, the box left from $(b_{k+1})_a$, if existent, contains an entry smaller than $e_s-1$. This finishes the induction step and thus the lemma is proved.
\end{proof}

\begin{proposition}\label{MinMax}
Let $T$ be a semistandard skew tableau of form $\mathcal{A}$ with $p$ columns satisfying the lattice property. Let $1 \leq k \leq p-1$ and let $h$ be the  height of the $k$-th column inside the body of $T$, i.e. the number of boxes in the $k$-th column of $T$ which lie in the body (so if $k = 1$ we do not count the boxes in the tail).

Then $\gamma_{p-k}(T) \geq h$ and if $k \geq 2$ then $\gamma_{p-k+2}(T) \leq h$.
\end{proposition}
\begin{proof}
Let $b$ be the lowest box in the $k$-th column and assume the entry of $b$ is $e$. 

Then just for the reason that the entries in a column are strictly increasing we have $h\le e$. By Lemma~\ref{AppearsAtLeast} the entry $e$ appears in $w(b)$ at least $p-k$ times, so each letter $1$, $2$ ..., $h$ appears at least $p-k$ times in $w(T)$ and hence $\gamma_{p-k}(T) \geq h$.

So assume $k \geq 2$ and assume $s = \gamma_{p-k+2}(T) > h$. By Lemma~\ref{AtMosth} the maximal possible entry in the box left from $b$, if this box exists, is $h$. So also an entry in a box to the left from $b$ which lies in the body is at most $h$ (note that there is no box below $b$ anywhere in the body, as $b$ is the lowest box in its column and $T$ of form $\mathcal{A}$). So inside the body an entry bigger than $h$ does not appear in the $1$st to $(k-1)$-th column. Hence each letter $h+1$, ..., $s$ must appear in the tail and in the $k$-th column and every column right from the $k$-th column, as otherwise it can not appear $p-k+2$ or more times in total. We will show that the $k$-th column then contains every letter $1$,..., $s$, which will yield the final contradiction, since clearly the $k$-th column can not contain more than $h$ entries. The proof is similar to the proof of Lemma~\ref{AtMosth}.

Let $b$ be the box in the $k$-th column containing $s$. Note that also $s$ can not appear in a box lying below and to the right of $b$, as $T$ is semistandard. Then $s$ appears in $w(b)$ exactly $p-k+1$ times, since it appears at least overall $p-k+2$ times in $w(T)$ and it does not appear in the $1$st to $(k-1)$-th column inside the body. The box left from $b$, if existent, contains an entry smaller than $h$, so also smaller than $s$, by Lemma~\ref{AtMosth}. We will show by decreasing induction that the $k$-th column contains every entry $e$ between $1$ and $s$ such that if $b_e$ is the box containing $e$ then $e$ appears exactly $p-k+1$ times in $w(b_e)$ and the box left from $b_e$, if existent, contains an entry smaller than $e$. By the above this holds for $e = s$, providing the base case of the induction.

So assume that this holds for a certain $e+1$ and let $b_{e+1}$ be the box in the $k$-th column containing $e+1$. Denote by $(b_{e+1})_a$, if it exists, the box above $b_{e+1}$ and by $(b_{e+1})_l$, if it exists, the box to the left of $b_{e+1}$. If $(b_{e+1})_l$ exists its entry is at most $e$, by induction. So if the box above $(b_{e+1})_l$ exists its entry is at most $e-1$. This implies that any entry lying to the left and not below $(b_{e+1})_a$ is smaller than $e$. Since by induction $e+1$ appears exactly $p-k+1$ times in $w(b_{e+1})$, also $e$ appears at least $p+k-1$ times in $w(b_{e+1})$. Combining this with the conclusion before, we see that $(b_{e+1})_a$ must exist and its entry must be $e$. Moreover the letter $e$ appears exactly $p-k+1$ times in $w((b_{e+1})_a)$ and the box to the left of $(b_{e+1})_a$ contains an entry smaller than $e$. This finishes the induction step and hence the proof.
\end{proof}

\section{Proofs of main results}

We proceed to prove Theorem~\ref{MainProp}. For the basic facts about Brauer trees we refer to \cite{HissLux}. Here we will only need that if $(K,R, k)$ is a $p$-modular system such that the characters labelling the vertices of the Brauer tree are afforded by $RG$-modules then each edge of the Brauer tree corresponds to a simple $kG$-module, each vertex of the Brauer tree corresponds to a simple $RG$-module with the given character and after reducing such a simple $RG$-module modulo the maximal ideal of $R$ the composition factors of this module are those corresponding to the adjacent edges, each with multiplicity $1$. 

\textit{Proof of Theorem~\ref{MainProp}:} Let $R$ be a complete discrete valuation ring containing a primitive $m$-th root of unity such that $R$ is an algebraic extension of $\mathbb{Z}_p$, representations affording the characters $\chi_1, \ ..., \ \chi_p$ can be realized over $R$ and $p$ is unramified in $R$. Such a ring $R$ exists by a result of Fong \cite[Remark 2.5]{4primaryII}. Let $k$ be the residue class field of $R$. We use the bar-notation to denote reduction modulo the maximal ideal of $R$ also with respect to modules. Let $M$ be a $k\langle u \rangle$-module. Adapting the notation from skew tableaus we denote by $\gamma_i(M)$ the number of direct indecomposable summands of $M$ of dimension at least $i$. 

Let $M'_1, \ ..., \ M'_p$ be $RG$-modules corresponding to the characters $\chi_1, \ ..., \ \chi_p$ and let, as explained in Proposition~\ref{PropChar0}, $\overline{M_1}, \ ..., \ \overline{M_p}$ be the direct summands of $\overline{M'_1}, ..., \ \overline{M'_p}$ respectively on which $u^p$ acts as $\xi^p$. Let $S'_1,\ ..., \ S'_{p-1}$ be the simple $kG$-modules corresponding to the edges of the Brauer tree in the natural order, i.e. $S'_i$ corresponds to the edge with vertices $\chi_i$ and $\chi_{i+1}$. Let $S_1, \ ..., \ S_{p-1}$ be the direct summand of $S'_1, \ ..., \ S'_{p-1}$ respectively on which $u^p$ acts as $\xi^p$. We will view $\overline{M_i}$ and $S_i$ as $k\langle u \rangle$-modules, unless explicitly stated otherwise. Let $\lambda_i$ be the partition corresponding to $\overline{M_i}$ and $\mu_i$ be the partition corresponding to $S_i$, in the sense of Section~\ref{sec:LatMet}. Note that by Proposition~\ref{PropChar0} if $S_j$ is a submodule of $\overline{M_i}$ then the skew diagram corresponding to $\lambda_i/\mu_j$ is of form $\mathcal{A}$ with $p$ columns, since the indecomposable summands of $\overline{M_i}$ are all of dimension $1$, $p-1$ or $p$, i.e. the possible length of rows in a Young diagram corresponding to $\lambda_i$ are $1$, $p-1$ and $p$. We will use the following fact several times.

\textit{Claim:} If we want to compute the possible isomorphism types of $S_i$ for some $1 < i < p-1$ using Theorem~\ref{ThLRCoeff} we can consider a semistandard skew tableau satisfying the lattice property of form $\lambda_i/\mu_{i-1}$ or $\lambda_{i+1}/\mu_{i+1}$. 

\textit{Proof of the claim:} The composition factors of $\overline{M'_i}$ as $kG$-module are $S'_{i-1}$ and $S'_i$ by the properties of Brauer trees. If $S'_{i-1}$ is a submodule of $\overline{M'_i}$ as $kG$-module then $S_i \cong \overline{M_i}/S_{i-1}$. So $S_i$ corresponds to a skew tableau of form $\lambda_i/\mu_{i-1}$ and $c^{\lambda_i}_{\mu_{i-1}, \mu_i} \neq 0$. If on the other hand $S'_i$ is a submodule of $\overline{M'_i}$ as $kG$-module then $S_{i-1}$ corresponds to a skew tableau of the form $\lambda_i/\mu_i$ and $c^{\lambda_i}_{\mu_i, \mu_{i-1}} \neq 0$. As the Littlewood-Richardson coefficient is symmetric this also implies $c^{\lambda_i}_{\mu_{i-1}, \mu_i} \neq 0$. So it suffices to consider $\lambda_i/\mu_{i-1}$, or rather the question if $c^{\lambda_i}_{\mu_{i-1}, \mu_i} \neq 0$. 

On the other hand we can consider $\overline{M'_{i+1}}$ as $kG$-module which has composition factors $S'_{i}$ and $S'_{i+1}$. If $S'_{i}$ is a submodule of $\overline{M'_i}$ as $kG$-module then $S_{i+1} \cong \overline{M_{i+1}}/S_{i}$. So $S_i$ corresponds to a skew tableau of form $\lambda_{i+1}/\mu_{i}$ and $c^{\lambda_{i+1}}_{\mu_{i}, \mu_{i+1}} \neq 0$. On the other hand, as in the case before, $S_{i}$ corresponds to a skew tableau of the form $\lambda_{i+1}/\mu_{i+1}$ and $c^{\lambda_{i+1}}_{\mu_{i+1}, \mu_{i}} \neq 0$. The symmetry of the Littlewood-Richardson coefficient implies that it suffices to consider if $c^{\lambda_{i+1}}_{\mu_{i+1}, \mu_i} \neq 0$. This finishes the proof of the claim.

We will next show that 
\begin{align}\label{ineq:geq}
\gamma_{p-2}(S_{p-2}) \geq \mu(\xi \cdot \zeta_p, u, \chi_{p-1}) - \mu(\xi \cdot \zeta_p, u, \chi_p)
\end{align}
 and also 
\begin{align}\label{ineq:leq}
\gamma_{p-2}(S_{p-2}) \leq \mu(\xi, u, \chi_1) - \sum_{i=1}^{p-2} (-1)^i \mu(\xi \cdot \zeta_p, u, \chi_i)
\end{align}

\noindent
which will imply the theorem.
First by the claim above the isomorphism type of $S_{p-2}$ can be described by a semistandard skew tableau $T$ satisfying the lattice property of shape $\lambda_{p-1}/\mu_{p-1}$. Note that $\mu_{p-1} = \lambda_p$ since $\overline{M_p} \cong S_{p-1}$, as $M_p$ corresponds to a vertex of the Brauer tree which has only one neighbour. By Proposition~\ref{MinMax} we know that $\gamma_{p-2}(T) = \gamma_{p-2}(S_{p-2})$ is at least as big as the height of the second column of $T$. Now by Proposition~\ref{PropChar0} the number of indecomposable direct summands of dimension at least $2$ in $\overline{M_{p-1}}$, i.e. the number of rows of length at least $2$ in a Young diagram corresponding to $\lambda_{p-1}$, is $\mu(\xi \cdot \zeta_p,u,\chi_{p-1})$ and the analogues statement holds for $\lambda_p$. So the height of the second column of $T$ is $\mu(\xi \cdot \zeta_p, u, \chi_{p-1}) - \mu(\xi \cdot \zeta_p, u, \chi_p)$, proving \eqref{ineq:geq}.

To prove \eqref{ineq:leq} we will prove by induction on $r$ that if $r$ is odd then 
$$\gamma_{r}(S_{r}) \leq \mu(\xi, u, \chi_1) - \sum_{i=1}^{r} (-1)^i \mu(\xi \cdot \zeta_p, u, \chi_i)$$
 and if $r$ is even then 
 $$\gamma_{p+1-r}(S_{r}) \geq -\mu(\xi, u, \chi_1) + \sum_{i=1}^{r} (-1)^i \mu(\xi \cdot \zeta_p, u, \chi_i).$$ 

Once this induction reaches $r = p-2$, which is an odd number, the first inequality will be exactly \eqref{ineq:leq} and so the proof of the theorem will be finished. 

For the base case consider $r = 1$. Then $\gamma_1(S_1)$ is the number of indecomposable direct summands of $S_1 \cong \overline{M_1}$ which by Proposition~\ref{PropChar0} is at most $\mu(\xi, u, \chi_1) + \mu(\xi \cdot \zeta_p, u, \chi_1)$. 

For the induction step assume that the statement holds for all numbers smaller than $r$. First assume that $r$ is even. Now $S_r$ corresponds to a semistandard skew tableau $T$ satisfying the lattice property of form $\lambda_r/\mu_{r-1}$. So by Proposition~\ref{MinMax} the number $\gamma_{p+1-r}(S_r) = \gamma_{p+1-r}(T)$ is at least as big as the height of the $(r-1)$-th column in the body of $T$. Since $r \leq p$ this numbers equals the difference of the number of direct indecomposable summands of $\overline{M_r}$ which is at least $(r-1)$-dimensional, but not $1$-dimensional, and the number of indecomposable direct summands of $S_{r-1}$ of dimension at least $r-1$, i.e. $\gamma_{r-1}(S_{r-1})$. The number of indecomposable direct summands of $\overline{M_r}$ which are at least $(r-1)$-dimensional, but not $1$-dimensional, is $\mu(\xi \cdot \zeta_p,u,\chi_r)$ by Proposition~\ref{PropChar0}.  Since by induction $\gamma_{r-1}(S_{r-1}) \leq \mu(\xi, u, \chi_1) - \sum_{i=1}^{r-1} (-1)^i \mu(\xi \cdot \zeta_p, u, \chi_i)$ we obtain that the height of the $(r-1)$-th column in the body of $T$ is at least 
$$ \mu(\xi \cdot \zeta_p,u,\chi_r) - (\mu(\xi, u, \chi_1) - \sum_{i=1}^{r-1} (-1)^i \mu(\xi \cdot \zeta_p, u, \chi_i)) = -\mu(\xi, u, \chi_1) + \sum_{i=1}^{r} (-1)^i \mu(\xi \cdot \zeta_p, u, \chi_i).$$

Now assume that $r$ is odd, bigger than $2$ and smaller than $p-1$. Also in this case $S_r$ corresponds to a semistandard skew tableau $T$ satisfying the lattice property of form $\lambda_r/\mu_{r-1}$. So the number $\gamma_r(S_r) = \gamma_r(T)$ is at most as big as the height of the $(p+2-r)$-th column in the body of $T$ by Proposition~\ref{MinMax}. This is the difference of the number of indecomposable direct summands of dimension at least $p+2-r$ in $\overline{M_r}$, i.e. $\gamma_{p+2-r}(\overline{M_r})$, and the number of indecomposable direct summands of $S_{r-1}$ of dimension at least $p+2-r$, i.e. $\gamma_{p+2-r}(S_{r-1}) = \gamma_{p+1-(r-1)}(S_{r-1})$. By Proposition~\ref{PropChar0} we know $\gamma_{p+2-r}(\overline{M_r}) = \mu(\xi \cdot \zeta_p, u, \chi_r)$ and by induction $\gamma_{p+1-(r-1)}(S_{r-1})$ is at least $-\mu(\xi, u, \chi_1) + \sum_{i=1}^{r-1} (-1)^i \mu(\xi \cdot \zeta_p, u, \chi_i)$. So the height of the $(p+2-r)$-th column of $T$ is at most 
$$\mu(\xi \cdot \zeta_p, u, \chi_r) - (-\mu(\xi, u, \chi_1) + \sum_{i=1}^{r-1} (-1)^i \mu(\xi \cdot \zeta_p, u, \chi_i)) = \mu(\xi, u, \chi_1) - \sum_{i=1}^{r} (-1)^i \mu(\xi \cdot \zeta_p, u, \chi_i).  $$
This finishes the proof of the induction claim and hence the proof of the theorem. \hfill \qed \\

The proof of the Prime Graph Question for the Conway simple groups now boils down to the application of Theorem~\ref{MainProp} for various cases described in \cite{BovdiKonovalovConway}. The information contained in \cite{BovdiKonovalovConway} is not completely sufficient for our purposes, since to compute multiplicities of eigenvalues we do not only need to know the partial augmentations of a unit $u \in \mathrm{V}(\mathbb{Z}G)$, but also the partial augmentations of its proper powers. These can be computed using the \texttt{GAP}-package \texttt{HeLP} \cite{HeLPPackage} and we will indicate in all cases we need to consider which characters are sufficient to obtain these possible partial augmentations using the command \texttt{HeLP\_WithGivenOrder}. Also in the case of the first Conway group we will use stronger results obtainable by the HeLP method than those given in \cite{BovdiKonovalovConway}.

We will denote by $\chi_{i}$ the $i$-th irreducible complex character of a group $G$ as given in the \texttt{GAP} character table library \cite{CTblLib}. We will also use names for conjugacy classes as in \cite{CTblLib}. The statements we will need about the $p$-blocks of various groups, their defect and their corresponding Brauer trees can all be derived from \texttt{GAP}, but they are also given in \cite{HissLux}.

\textit{Proof of Theorem~\ref{TheoremConway}:} We will study the three cases of interest separately.

\textbf{Case $G = Co_3$:} The information on the orders of elements in $G$ is contained in \cite{ATLAS}. Using this information and \cite[Theorem 1(i)]{BovdiKonovalovConway} we only have to consider units of order $35$ in $\mathrm{V}(\mathbb{Z}G)$. Using $\chi_2$ to compute partial augmentations for units of order $5$ and $\chi_2$, $\chi_3$ to compute the partial augmentations of units of order $35$ we are left with the following possibilities:
$$(\varepsilon_{5a}(u^7), \varepsilon_{5b}(u^7),\varepsilon_{5a}(u),\varepsilon_{5b}(u),\varepsilon_{7a}(u)) \in \{(-4,5,3,12,-14), (-3,4,4,11,-14) \}. $$
Note that there is only one class of elements of order $7$ in $G$, so we do not need to consider the partial augmentations of $u^5$, since the only class in which a partial augmentation at $u^5$ is non-vanishing is $7a$.

$G$ possesses a $5$-block of defect $1$ whose Brauer tree is a line of form
\[
\begin{tikzpicture}
\node[label=north:{$\chi_5$}] at (0,1.5) (1){};
\node[label=north:{$\chi_{29}$}] at (1.5,1.5) (2){};
\node[label=north:{$\chi_{39}$}] at (3,1.5) (3){};
\node[label=north:{$\chi_{35}$}] at (4.5,1.5) (4){};
\node[label=north:{$\chi_{12}$}] at (6,1.5) (5){};
\foreach \p in {1,2,3,4, 5}{
\draw[fill=white] (\p) circle (.075cm);
}
\draw (.075,1.5)--(1.425,1.5);
\draw (1.575,1.5)--(2.925,1.5);
\draw (3.075,1.5)--(4.425,1.5);
\draw (4.575,1.5)--(5.925,1.5);
\end{tikzpicture}.
\] 
All irreducible ordinary characters in this block have only integral values and $5$ is of course not ramified in $\mathbb{Z}$. So we can apply Theorem~\ref{MainProp}. We provide the multiplicities of the needed eigenvalues of $u$, for the two critical distributions of partial augmentations, in Table~\ref{TableCo3}. Here the first entry in each column contains the possible values of $(\varepsilon_{5a}(u^7), \varepsilon_{5b}(u^7),\varepsilon_{5a}(u),\varepsilon_{5b}(u),\varepsilon_{7a}(u))$.

\begin{table}[h]
\begin{tabular}{l|ll|ll}
 & \multicolumn{2}{c|}{$(-4,5,3,12,-14)$} & \multicolumn{2}{c}{$(-3,4,4,11,-14)$}  \\ \hline \hline
 $\chi$ & $\mu(1,u,\chi)$ & $\mu(\zeta_5,u,\chi)$ & $\mu(1,u,\chi)$ & $\mu(\zeta_5,u,\chi)$ \\ \hline
 $\chi_5$ & 33 & 2 & 29 & 3 \\
 $\chi_{29}$ &   & 2119 &  & 2118 \\
 $\chi_{39}$ &   & 7029 &  & 7030 \\
$\chi_{35}$ & & 5071 &  & 5070 \\
$\chi_{12}$  &  & 104 &  & 105 
\end{tabular}
\caption{\label{TableCo3} Multiplicities of eigenvalues for $G = Co_3$ for units of order $35$.
	}
\end{table} 

So using Theorem~\ref{MainProp} with $\xi = 1$ we would have, e.g. if the partial augmentations are $(\varepsilon_{5a}(u^7), \varepsilon_{5b}(u^7),\varepsilon_{5a}(u),\varepsilon_{5b}(u),\varepsilon_{7a}(u)) = (-4,5,3,12,-14)$ that
$$\mu(\zeta_5,u,\chi_{35}) - \mu(\zeta_5,u,\chi_{12}) = 4967 \leq \mu(1,u,\chi_{5}) + \mu(\zeta_5,u,\chi_{5}) - \mu(\zeta_5,u,\chi_{29}) + \mu(\zeta_5,u,\chi_{39}) = 4945, $$
a contradiction. The same argument applies to the other possibility of distributions of partial augmentations of $u$, in which case we get $4965 \leq 4944$.

\textbf{Case $G = Co_2$:}  The information on the orders of elements in $G$ is contained in \cite{ATLAS}. Using this information and \cite[Theorem 2(i)]{BovdiKonovalovConway} we only have to consider units of order $35$ in $\mathbb{Z}G$. Using $\chi_2$ to compute partial augmentations of units of order $5$ and $\chi_2$, $\chi_3$ for order $35$ we get
$$(\varepsilon_{5a}(u^7), \varepsilon_{5b}(u^7),\varepsilon_{5a}(u),\varepsilon_{5b}(u),\varepsilon_{7a}(u)) \in \{(-4,5,3,12,-14), (-3,4,4,11,-14) \}. $$
Again there is only one class of elements of order $7$, so we do not need to consider $u^5$.

Also in this case $G$ possesses a $5$-block of defect 1 with Brauer tree of the form
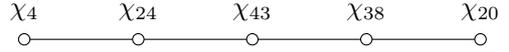
\[
\begin{tikzpicture}
\node[label=north:{$\chi_4$}] at (0,1.5) (1){};
\node[label=north:{$\chi_{24}$}] at (1.5,1.5) (2){};
\node[label=north:{$\chi_{43}$}] at (3,1.5) (3){};
\node[label=north:{$\chi_{38}$}] at (4.5,1.5) (4){};
\node[label=north:{$\chi_{20}$}] at (6,1.5) (5){};
\foreach \p in {1,2,3,4, 5}{
\draw[fill=white] (\p) circle (.075cm);
}
\draw (.075,1.5)--(1.425,1.5);
\draw (1.575,1.5)--(2.925,1.5);
\draw (3.075,1.5)--(4.425,1.5);
\draw (4.575,1.5)--(5.925,1.5);
\end{tikzpicture}.
\]

All ordinary characters in this block only have integral values.
The necessary multiplicities are provided in Table~\ref{TableCo2}

\begin{table}[h]
\begin{tabular}{l|ll|ll}
 & \multicolumn{2}{c|}{$(-4,5,3,12,-14)$} & \multicolumn{2}{c}{$(-3,4,4,11,-14)$}  \\ \hline \hline
 $\chi$ & $\mu(1,u,\chi)$ & $\mu(\zeta_5,u,\chi)$ & $\mu(1,u,\chi)$ & $\mu(\zeta_5,u,\chi)$ \\ \hline
 $\chi_5$ & 33 & 2 & 29 & 3 \\
 $\chi_{29}$ &   & 3269 & & 3268 \\
 $\chi_{39}$ &  & 13354 &  & 13355 \\
$\chi_{35}$ & & 11396 &  & 11395 \\
$\chi_{12}$  &   & 1254  &   & 1255
\end{tabular}
\caption{\label{TableCo2} Multiplicities of eigenvalues for $G = Co_2$ for units of order $35$.
	}
\end{table} 

In the first case using Theorem~\ref{MainProp} with $\xi = 1$ we get $10142 \leq 10120$ and in the second case $10140 \leq 10119$.

\textbf{Case $G=Co_1$:}  The information on the orders of elements in $G$ is contained in \cite{ATLAS}. Using this information and \cite[Theorem 3(i)]{BovdiKonovalovConway} we only have to consider units of order $55$ and $65$ in $\mathbb{Z}G$. We will first apply the HeLP method here to obtain stronger results than in \cite{BovdiKonovalovConway}. The $2$-, $3$- and $5$-modular Brauer tables are not available in \texttt{GAP}, but the online version of the Atlas, also implemented as a \texttt{GAP}-package \cite{AtlasRep}, contains several representations in these characteristics and we will use a character $\psi$ coming from a $24$-dimensional representation of $G$ over $\mathbb{F}_2$. The values of $\psi$ on classes of interest can be found in Table~\ref{TablePsi}.

\begin{table}[h]
\begin{tabular}{c|c|c|c|c|c|c|c} 
  & $1a$ & $5a$ & $5b$ & $5c$ & $11a$ & $13a$ \\ \hline
  $\psi$ & 24 & $-6$ & $4$ & $-1$ & $2$ & $-2$
\end{tabular}
\caption{\label{TablePsi} A $2$-modular Brauer character for $Co_1$.
	}
\end{table}

The results we obtain using the HeLP method are stronger than those in \cite{BovdiKonovalovConway}, since we include $\psi$. Using $\chi_{2}$, $\chi_{3}$ and $\psi$ we obtain $98$ possibilities for the partial augmentations of units of order $5$. Then using $\chi_{2}$, $\chi_{4}$, $\chi_{5}$ and $\psi$ we can exclude the existence of units of order $65$ in $\mathbb{Z}G$. For units of order $55$ we apply the characters $\chi_{2}$, $\chi_{3}$, $\chi_{4}$ and $\psi$. Note that there is only one class of elements of order $11$ in $G$.

We get that $(\varepsilon_{5a}(u^{11}), \varepsilon_{5b}(u^{11}), \varepsilon_{5c}(u^{11}), \varepsilon_{5a}(u), \varepsilon_{5b}(u), \varepsilon_{5c}(u), \varepsilon_{11a}(u) )$ is one of the four possibilities
$$  (1, 5, -5 , 1, -6, -5, 11), (1, 6, -6, 1, -5, -6, 11 ), ( 2, 6, -7, 2, -5, -7, 11), (2, 7, -8 , 2, -4, -8, 11).  $$

The principal $11$-block of $G$ has a Brauer tree of the form

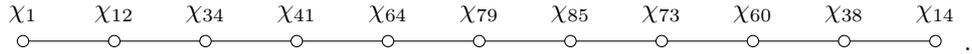
\[
\begin{tikzpicture}
\node[label=north:{$\chi_1$}] at (0,1.5) (1){}; 
\node[label=north:{$\chi_{12}$}] at (1.2,1.5) (2){};
\node[label=north:{$\chi_{34}$}] at (2.4,1.5) (3){};
\node[label=north:{$\chi_{41}$}] at (3.6,1.5) (4){};
\node[label=north:{$\chi_{64}$}] at (4.8,1.5) (5){};
\node[label=north:{$\chi_{79}$}] at (6,1.5) (6){};
\node[label=north:{$\chi_{85}$}] at (7.2,1.5) (7){};
\node[label=north:{$\chi_{73}$}] at (8.4,1.5) (8){};
\node[label=north:{$\chi_{60}$}] at (9.6,1.5) (9){};
\node[label=north:{$\chi_{38}$}] at (10.8,1.5) (10){};
\node[label=north:{$\chi_{14}$}] at (12,1.5) (11){};
\foreach \p in {1,2,3,4,5,6,7,8,9,10,11}{
\draw[fill=white] (\p) circle (.075cm);
}
\draw (.075,1.5)--(1.125,1.5);
\draw (1.275,1.5)--(2.325,1.5);
\draw (2.475,1.5)--(3.525,1.5);
\draw (3.675,1.5)--(4.725,1.5);
\draw (4.875,1.5)--(5.925,1.5);
\draw (6.075,1.5)--(7.125,1.5);
\draw (7.275,1.5)--(8.325,1.5);
\draw (8.475,1.5)--(9.525,1.5);
\draw (9.675,1.5)--(10.725,1.5);
\draw (10.875,1.5)--(11.925,1.5);
\end{tikzpicture}.
\]

Also in this block all ordinary characters only have integral values.
The necessary multiplicities are provided in Table~\ref{TableCo1}.
{\footnotesize
\begin{table}[h]
\begin{tabular}{l|ll|ll|ll|ll} 
 & \multicolumn{2}{c|}{$(1, 5, -5 , 1, -6, -5, 11)$} &  \multicolumn{2}{c|}{$(1, 5, -5 , 1, -6, -5, 11)$} & \multicolumn{2}{c|}{$(1, 5, -5 , 1, -6, -5, 11)$} & \multicolumn{2}{c}{$(1, 5, -5 , 1, -6, -5, 11)$} \\ \hline \hline
 $\chi$ & $\mu(1,u,\chi)$ & $\mu(\zeta_{11},u,\chi)$ & $\mu(1,u,\chi)$ & $\mu(\zeta_{11},u,\chi)$ & $\mu(1,u,\chi)$ & $\mu(\zeta_{11},u,\chi)$  & $\mu(1,u,\chi)$ & $\mu(\zeta_{11},u,\chi)$\\ \hline
 $\chi_1$ & 1 & 0 & 1 & 0 & 1 & 0 & 1 & 0  \\
 $\chi_{12}$ &  & 5668 & & 5670  & & 5638 & & 5640 \\
 $\chi_{34}$ &  & 138600 & & 138600 & & 138520  & & 138520  \\
$\chi_{41}$ &  & 391385 & &  391385 & & 391350  & & 391350 \\
$\chi_{64}$  &  & 1929876 & &  1929870 & & 1929856  & & 1929850 \\
$\chi_{79}$  &  & 4495195 & &  4495195 & & 4495195  & & 4495195 \\
$\chi_{85}$  &  & 5326485 & &  5326495  & & 5326570  & & 5326580 \\
$\chi_{73}$  &  & 3734505 & & 3734505 & & 3734540  & & 3734540 \\
$\chi_{60}$  &  & 1522180 & & 1522180  & &  1522180  & & 1522180 \\
$\chi_{38}$  &  & 297293  & & 297285 & & 297303  & & 297295 \\
$\chi_{14}$  &  & 6885 & & 6875 & & 6880  & & 6870
\end{tabular}
\caption{\label{TableCo1} Multiplicities of eigenvalues for $G = Co_1$ for units of order $55$.
	}
\end{table}
}

From the four possibilities we obtain by Theorem~\ref{MainProp} the inequalities
$290408 \leq 290389$, $290410 \leq 290391$, $290423 \leq 290304$ and $290425 \leq 290306$, all of which do not hold. \hfill \qed

\begin{remark} In the proof of Theorem~\ref{TheoremConway} to handle the first Conway group $Co_1$ we used Theorem~\ref{MainProp} with $p=11$ instead of $p=5$ for two reasons. One is to show that Theorem~\ref{MainProp} can also be applied with other primes than $5$. The other, more importantly, is that the $11$-Brauer trees can be computed using \texttt{GAP} which can not be done as easily with $5$-Brauer trees, since the $5$-Brauer table is not available in \texttt{GAP}. The 5-Brauer trees are given in \cite{HissLux}, but the calculations with $p=11$ are easier to verify for a reader who does not have \cite{HissLux} available. Using the first Brauer tree given in \cite[Section 6.22]{HissLux} and Theorem~\ref{MainProp} with $p = 5$ and $\xi = 1$ one can also get the same result. 
\end{remark}

\begin{remark} Of the $17$ sporadic groups studied by Bovdi, Konovalov et al. the only group for which they could not prove the Prime Graph Question, apart from the Conway groups, was the O'Nan simple group \cite{BovdiKonovalovONan}. Here it remains to exclude the existence of units of order $33$ and $57$. Theorem~\ref{MainProp} can be applied here with $p=11$ to handle units of order $33$, but not for order $57$. This is not possible since for $p \in \{3,19 \}$ every $p$-block of $G$ which is a Brauer Tree Algebra contains characters such that $p$ is ramified in the ring of values of these characters.
\end{remark}

\section*{Acknowledgements}
The setting handled in this paper was discussed at various stages with Alexander Konovalov, Andreas Bächle, Alexander Thumm and Eugenio Giannelli. I am thankful to them all for their input.

I would also like to thank the referee whose comments helped to improve the readability, particularly of the combinatorial proofs, significantly.

\bibliographystyle{amsalpha}
\bibliography{ConwayPQ}

\newcommand{\etalchar}[1]{$^{#1}$}
\providecommand{\bysame}{\leavevmode\hbox to3em{\hrulefill}\thinspace}
\providecommand{\MR}{\relax\ifhmode\unskip\space\fi MR }
\providecommand{\MRhref}[2]{%
  \href{http://www.ams.org/mathscinet-getitem?mr=#1}{#2}
}
\providecommand{\href}[2]{#2}
\begin{thebibliography}{WPN{\etalchar{+}}11}

\bibitem[BC17]{BaechleCaicedoSymmetric}
A.~B\"achle and M.~Caicedo, \emph{On the prime graph question for almost simple
  groups with an alternating socle}, Internat. J. Algebra Comput. \textbf{27}
  (2017), no.~3, 333--347.

\bibitem[BGK09]{BovdiKonovalovONan}
V.~Bovdi, A.~Grishkov, and A.~Konovalov, \emph{Kimmerle conjecture for the
  {H}eld and {O}'{N}an sporadic simple groups}, Sci. Math. Jpn. \textbf{69}
  (2009), no.~3, 353--361.

\bibitem[BJK11]{BovdiKonovalovJ}
V.~Bovdi, E.~Jespers, and A.~Konovalov, \emph{Torsion units in integral group
  rings of {J}anko simple groups}, Math. Comp. \textbf{80} (2011), no.~273,
  593--615.

\bibitem[BK07a]{BovdiKonovalovM11}
V.~Bovdi and A.~Konovalov, \emph{Integral group ring of the first {M}athieu
  simple group}, Groups {S}t. {A}ndrews 2005. {V}ol. 1, London Math. Soc.
  Lecture Note Ser., vol. 339, Cambridge Univ. Press, Cambridge, 2007,
  pp.~237--245.

\bibitem[BK07b]{BovdiKonovalovMcL}
\bysame, \emph{Integral group ring of the {M}c{L}aughlin simple group}, Algebra
  Discrete Math. (2007), no.~2, 43--53.

\bibitem[BK08]{BovdiKonovalovM23}
\bysame, \emph{Integral group ring of the {M}athieu simple group {$M_{23}$}},
  Comm. Algebra \textbf{36} (2008), no.~7, 2670--2680.

\bibitem[BK09]{BovdiKonovalovRu}
\bysame, \emph{Integral group ring of {R}udvalis simple group}, Ukra\"\i n.
  Mat. Zh. \textbf{61} (2009), no.~1, 3--13.

\bibitem[BK10]{BovdiKonovalovHS}
\bysame, \emph{Torsion units in integral group ring of {H}igman-{S}ims simple
  group}, Studia Sci. Math. Hungar. \textbf{47} (2010), no.~1, 1--11.

\bibitem[BK12]{BovdiKonovalovM24}
\bysame, \emph{Integral group ring of the {M}athieu simple group {$M_{24}$}},
  J. Algebra Appl. \textbf{11} (2012), no.~1, 1250016, 10.

\bibitem[BKL08]{BovdiKonovalovM22}
V.~Bovdi, A.~Konovalov, and S.~Linton, \emph{Torsion units in integral group
  ring of the {M}athieu simple group {${\rm M}_{22}$}}, LMS J. Comput. Math.
  \textbf{11} (2008), 28--39.

\bibitem[BKL11]{BovdiKonovalovConway}
\bysame, \emph{Torsion units in integral group rings of {C}onway simple
  groups}, Internat. J. Algebra Comput. \textbf{21} (2011), no.~4, 615--634.

\bibitem[BKM08]{BovdiKonovalovSuz}
V.~Bovdi, A.~Konovalov, and E.~Marcos, \emph{Integral group ring of the
  {S}uzuki sporadic simple group}, Publ. Math. Debrecen \textbf{72} (2008),
  no.~3-4, 487--503.

\bibitem[BKS07]{BovdiKonovalovM12}
V.~Bovdi, A.~Konovalov, and S.~Siciliano, \emph{Integral group ring of the
  {M}athieu simple group {$M_{12}$}}, Rend. Circ. Mat. Palermo (2) \textbf{56}
  (2007), no.~1, 125--136.

\bibitem[BM17a]{HeLPPackage}
A.~B\"achle and L.~Margolis, \emph{{HeLP} -- {H}ertweck-{L}uthar-{P}assi
  method, \textsf{GAP} package, {V}ersion 3.1},
  \url{http://homepages.vub.ac.be/abachle/help/}, 2017.

\bibitem[BM17b]{4primaryI}
\bysame, \emph{On the prime graph question for integral group rings of
  4-primary groups {I}}, Internat. J. Algebra Comput. \textbf{27} (2017),
  no.~6, 731--767.

\bibitem[BM17c]{Gitter}
\bysame, \emph{Rational {C}onjugacy of {T}orsion {U}nits in {I}ntegral {G}roup
  {R}ings of {N}on-{S}olvable {G}roups}, Proc. Edinb. Math. Soc. (2)
  \textbf{60} (2017), no.~4, 813--830.

\bibitem[BM18a]{HeLPPaper}
\bysame, \emph{He{LP}: a {GAP} package for torsion units in integral group
  rings}, J. Softw. Algebra Geom. \textbf{8} (2018), 1--9.

\bibitem[BM18b]{4primaryII}
\bysame, \emph{On the {P}rime {G}raph {Q}uestion for {I}ntegral {G}roup {R}ings
  of 4-primary groups {II}}, Algebr. Represent. Theory (2018), 1--21, in press,
  doi: 10.1007/s10468-018-9776-6.

\bibitem[Bre12]{CTblLib}
T.~Breuer, \emph{The \textsf{GAP} {C}haracter {T}able {L}ibrary, {V}ersion
  1.2.1}, \url{http://www.math.rwth-aachen.de/\~Thomas.Breuer/ctbllib}, May
  2012, \textsf{GAP} package.

\bibitem[CCN{\etalchar{+}}85]{ATLAS}
J.~H. Conway, R.~T. Curtis, S.~P. Norton, R.~A. Parker, and R.~A. Wilson,
  \emph{Atlas of finite groups}, Oxford University Press, Eynsham, 1985,
  Maximal subgroups and ordinary characters for simple groups, With
  computational assistance from J. G. Thackray.

\bibitem[CL65]{CohnLivingstone}
J.~A. Cohn and D.~Livingstone, \emph{On the structure of group algebras. {I}},
  Canad. J. Math. \textbf{17} (1965), 583--593.

\bibitem[EM18]{EiseleMargolis17}
F.~Eisele and L.~Margolis, \emph{A counterexample to the first {Z}assenhaus
  conjecture}, Adv. Math. \textbf{339} (2018), 599--641.

\bibitem[Ful97]{Fulton}
W.~Fulton, \emph{Young tableaux}, London Mathematical Society Student Texts,
  vol.~35, Cambridge University Press, Cambridge, 1997, With applications to
  representation theory and geometry.

\bibitem[GAP17]{GAP}
The GAP~Group, \emph{{GAP -- Groups, Algorithms, and Programming, Version
  4.8.8}}, 2017, \url{http://www.gap-system.org}.

\bibitem[Her01]{HertweckAnnals}
M.~Hertweck, \emph{A counterexample to the isomorphism problem for integral
  group rings}, Ann. of Math. (2) \textbf{154} (2001), no.~1, 115--138.

\bibitem[Her07]{HertweckBrauer}
\bysame, \emph{Partial augmentations and {B}rauer character values of torsion
  units in group rings}, arXiv:0612429v2 [math.RA], 2004 - 2007.

\bibitem[HL89]{HissLux}
G.~Hiss and K.~Lux, \emph{Brauer trees of sporadic groups}, Oxford Science
  Publications, The Clarendon Press, Oxford University Press, New York, 1989.

\bibitem[JdR16]{JespersDelRioGRG}
E.~Jespers and \'A. del R{\'\i}o, \emph{Group ring groups. {V}ol. 1. {O}rders
  and generic constructions of units}, De Gruyter Graduate, De Gruyter, Berlin,
  2016.

\bibitem[Kim06]{KimmiPQ}
W.~Kimmerle, \emph{On the prime graph of the unit group of integral group rings
  of finite groups}, Groups, rings and algebras, Contemp. Math., vol. 420,
  Amer. Math. Soc., Providence, RI, 2006, pp.~215--228.

\bibitem[KK15]{KimmiKonovalov15}
W.~Kimmerle and A.~Konovalov, \emph{Recent advances on torsion subgroups of
  integral group rings}, Groups {S}t {A}ndrews 2013, London Math. Soc. Lecture
  Note Ser., vol. 422, Cambridge Univ. Press, Cambridge, 2015, pp.~331--347.

\bibitem[KK17]{KimmiKonovalov17}
\bysame, \emph{On the {G}ruenberg--{K}egel graph of integral group rings of
  finite groups}, Internat. J. Algebra Comput. \textbf{27} (2017), no.~6,
  619--631.

\bibitem[LP89]{LutharPassi1989}
I.~S. Luthar and I.~B.~S. Passi, \emph{Zassenhaus conjecture for {$A_5$}},
  Proc. Indian Acad. Sci. Math. Sci. \textbf{99} (1989), no.~1, 1--5.

\bibitem[Seh93]{Sehgal1993}
S.~K. Sehgal, \emph{Units in integral group rings}, Pitman Monographs and
  Surveys in Pure and Applied Mathematics, vol.~69, Longman Scientific \&
  Technical, Harlow, 1993.

\bibitem[WPN{\etalchar{+}}11]{AtlasRep}
R.A. Wilson, R.A. Parker, S.~Nickerson, J.N. Bray, and T.~Breuer,
  \emph{{AtlasRep}, {A} \textsf{GAP} {I}nterface to the {A}tlas of {G}roup
  {R}epresentations, {V}ersion 1.5},
  \url{http://www.math.rwth-aachen.de/\~Thomas.Breuer/atlasrep}, July 2011,
  Refereed \textsf{GAP} package.

\end{thebibliography}

\end{document}